\documentclass[11pt]{article}

\usepackage{graphicx, subfigure, amsmath, amssymb}
\usepackage[figuresright]{rotating}
\input{epsf.sty}  

\newtheorem{lemma}{Lemma}
\newtheorem{theorem}{Theorem}
\newtheorem{definition}{Definition}
\newtheorem{example}{Example}
\newtheorem{remark}{Remark}
\newtheorem{proposition}{Proposition}
\newtheorem{corollary}{Corollary} 
\newenvironment{proof}{\begin{trivlist}
\item[\hspace{\labelsep}{\bf\noindent Proof. }]}
{
\end{trivlist}
}

\usepackage{color}

\title{\bf Extension of the past lifetime \\ 
and its connection to the cumulative entropy} 

\author{
{\sc Antonio Di Crescenzo}\footnote{
Dipartimento di Matematica, Universit\`a degli Studi di Salerno,
Via Giovanni Paolo II n.\ 132, I-84084 Fisciano (SA), Italy, 
Email: adicrescenzo@unisa.it 
} 
, \ 
{\sc Abdolsaeed Toomaj}\footnote{
Department of Statistics, Gonbad Kavous University,
Gonbad Kavous, Iran, 
Email: ab.toomaj@gmail.com
} 
}
\date{\normalsize 
\bf First published in {\em Journal of Applied Probability}  \\
Vol.\ 52, p.\ 1156--1174  \ \copyright\ 2015 by Applied Probability Trust}

\begin{document}
\maketitle

\begin{abstract}
Given two absolutely continuous nonnegative independent random variables, we define the reversed relevation 
transform as dual to the relevation transform. We first apply such transforms to the lifetimes of the components 
of parallel and series systems under suitably proportionality assumptions on the hazards rates. Furthermore, we 
prove that the (reversed) relevation transform is commutative if and only if the proportional (reversed) hazard 
rate model holds. By repeated application of the reversed relevation transform we construct a decreasing 
sequence of random variables which leads to new weighted probability densities. We obtain various relations 
involving ageing notions and stochastic orders. We also exploit the connection of such a sequence to the 
cumulative entropy and to an operator that is dual to the Dickson--Hipp operator. Iterative formulae for computing 
the mean and the cumulative entropy of the random variables of the sequence are finally investigated. 



\end{abstract}

\section{Introduction}\label{intro}
In this paper we introduce the reversed relevation transform and study some properties of a new weighted cumulative distribution function and its connection with the cumulative entropy. 
The considered stochastic model is dual to the weighted tail distribution studied by Kapodistria and Psarrakos \cite{Kapo-Psarr}. Specifically, we construct a sequence of stochastically decreasing random variables 
$\{X_n,n\geq1\}$. In this sequence, $X_1$ is  nonnegative and absolutely continuous, and the $(n+1)$th random variable of the sequence  is inductively defined  through the following relation: 
$[X_{n+1}\,|\,X_n=t]\stackrel{D}{=}[X_n\,|\,X_n\leq t],\ n=1,2,\cdots,$ for $t>0$. Here, as usual, $[X\,|\,B]$ denotes a random variable having the same distribution of $X$ conditional on $B$, and $\stackrel{D}{=}$ denotes equality in distribution. 
\par
Roughly speaking, $\{X_n,n\geq1\}$ is suitable to describe an iterative process involving a sequence of tasks, where 
$X_n$ is the random time required to perform the $n$th task. For instance, we refer to a training procedure based on iterative learning or a working system based on replacements or repairs of failed items. 
Starting from the sequence $\{X_n,n\geq1\}$, we derive some properties of a new weighted cumulative distribution via stochastic orders, and its connection with covariance and cumulative entropy. 
Our investigation is also devoted to disclosing iterative rules that allow us to compute the mean and the cumulative entropy of the random variables of the sequence, whose computational efficiency is illustrated by some numerical examples. 
\par
We recall that   \cite{Kapo-Psarr} constructed a stochastically increasing sequence of random variables, whose iterative rule involved the residual lifetime of  $X_n,$ i.e.
$$
 [X_{n+1}-X_{n}\,|\,X_n=t]\stackrel{d}{=}[X_{n}-t\,|\,X_n>t],\qquad n\geq 1,\ t>0. 
$$
They obtained some results  on this sequence and its connections to the cumulative residual entropy. 
Their process may 
describe the successive failures of a component, which, on failure, is replaced by a component of equal age, but the lifetime distribution of the $n$th component is assumed to be identical to the distribution of the time until the $n$th failure; see \cite{Kapo-Psarr} for more details. A simpler case was studied by Baxter \cite{Baxter}, who considered a stochastic process generated by the successive failures of a component which on failure is replaced by a 
component of equal age.
\par
The paper is organized as follows. In Section 2 we present the reversed relevation transform with a 
preliminary result based on the usual stochastic order and an application to parallel systems involving the 
proportional reversed hazards rate model (PRHRM). Some dual results for the relevation transform are then provided.  
We also address the problem of determining necessary and sufficient conditions such that the 
reversed relevation transform and the relevation transform are commutative. 
The new weighted distributions and their characteristics based on stochastic orders and ageing properties 
are discussed in Section 3, where the new notion of decreasing reversed hazard rate (DRHR) in a length-biased sense is also considered. In Section 4 various kinds of entropy such as Shannon entropy, cumulative entropy and dynamic cumulative entropy are examined. Then the connections between the earlier mentioned entropies and several functions of the given sequence are discussed. 
Specifically, we also obtain some iterative results for the involved quantities. These include a new probabilistic meaning for the cumulative entropy, which can be expressed as a difference of means of consecutive random variables of the  considered sequence. 
Finally, in Section 5 we define an integral operator and we discuss its properties related to the previous results. Also, various numerical examples are presented to shed further light on the characteristics of the studied sequence.
\section{Background and preliminary results} 
Let $X$ be an absolutely continuous nonnegative random variable with probability density function (PDF) $f(t)$, cumulative distribution function (CDF) $F(t)=\mathbb{P}(X\leq t)$, and survival function $\bar{F}(t)=1-F(t)$, so that $X$ may be viewed as the random lifetime of a system or a component or a living organism. Assume that $F(t)>0$ for all $t>0$. 
We recall that the reversed hazard rate function of $X$ is defined by
$$
\tau(t)=\frac{d}{dt}\log F(t)=\frac{f(t)}{F(t)},\qquad t>0.
$$
There are several papers on applications of reversed hazard rate function in the literature, see e.g.\ Block \emph{et al.}\ \cite{Block}, Di Crescenzo \cite{Di-Crescenzo-2000}, Gupta and Nanda \cite{Gupta-Nanda-2001}, Kijima and Ohnishi \cite{Kijima-Ohnishi-1999}, and the references therein. For the random lifetime $X$, we define $X_{[t]}=[t-X|X\leq t],\ t>0$, 
which is termed the inactivity time. Indeed, $X_{[t]}$ describes the length of the time interval occurring between the failure time $X$ and an inspection time $t$, given that at time $t$ the system has been found failed. For $t>0$, the mean inactivity time of $X$ is given by
\begin{eqnarray}\label{def:mean:residual}
\tilde{\mu}(t)=\mathbb E[t-X\,|\,X\leq t]=\frac{1}{F(t)}\int_{0}^{t}F(x)dx.
\end{eqnarray}
It is known that the CDF of the past lifetime $[X\,|\,X\leq t],\ t>0,$ is given by
$$
P(X\leq x|X\leq t)=\left\{
\begin{array}{lcl}
\displaystyle\frac{F(x)}{F(t)} \ &~~& \ 0\leq x\leq t\\
1 \ &~~~~& \ x>t,
\end{array}
\right.
$$
so that the PDF of the past lifetime is $f(x)/F(t)$ for all $0<x<t$.
\par
Hereafter, we consider two nonnegative absolutely continuous and independent random 
variables $X$ and $Y$ with the CDFs $F(t)$ and $G(t)$, respectively.
\begin{definition}\label{def:2.1}\rm
The reversed relevation transform of $X$ and $Y$ is defined by:
\begin{eqnarray}\label{G:symbol:F}
G\widetilde{\#}F(x)&=&
\int_{0}^{\infty}\bigg\{\frac{F(x)}{F(t)}{\bf 1}_{\{0\leq x\leq t\}}+{\bf 1}_{\{x>t\}}\bigg\}{\rm d}G(t)\nonumber\\
&=&G(x)+F(x)\int_{x}^{\infty}\frac{1}{F(t)}{\rm d}G(t),
\qquad x>0,
\end{eqnarray}
where ${\bf 1}_A$ is the indicator function of the set $A$, i.e.\ ${\bf 1}_A(x)=1$ if $x\in A$, and ${\bf 1}_A(x)=0$ if $x\in A^{c}$.
\end{definition}
\par
Generally, the inactivity time of $X$ at a random time $Y$, denoted by $X_{[Y]}$, 
is defined as $X_{[Y]}\stackrel{D}{=}[Y-X\,|\,X\leq Y]$ (see, e.g.\ \cite{Kapo-Psarr}). 
Moreover, let $X[Y]\stackrel{D}{=}[X\,|\,X\leq Y]$ denote the total time of $X$ given that it is less than an independent 
random inspection time $Y$. Therefore, the CDF of $X[Y]$ is given by  
\begin{eqnarray}\label{X[Y]}
\mathbb P(X[Y]\leq x)=G\widetilde{\#}F(x),
\end{eqnarray}
where the symbol $\widetilde{\#}$ is defined in   (\ref{G:symbol:F}). 
If random variables $X$ and $Y$ are independent and identically distributed (i.i.d.), then  
$$
 \mathbb P(X[Y]\leq x)=F\widetilde{\#}F(x)=F(x)[1+T(x)],\qquad x>0,
$$
where
\begin{equation}\label{cumulative:reversed:hazard}
T(x)=-\log F(x)=\int_{x}^{\infty}\tau(u){\rm d}u,\qquad x>0,
\end{equation}
denotes the cumulative reversed hazard rate function of $X$; see, e.g.\  \cite{Di-Crescenzo-2000}.
\begin{example}\label{eq:FGab}\rm
(a) Let $X$ and $Y$ be independent nonnegative random variables having the cdfs $F(x)=\exp(-ax^{-\gamma}),\ x>0$, 
and $G(x)=\exp(-cx^{-\gamma}),\ x>0$, respectively, with $a,c, \gamma>0$. From (\ref{G:symbol:F}) and (\ref{X[Y]}), 
we have
$$
\mathbb P(X[Y]\leq x)=G\widetilde{\#}F(x)=\left\{
\begin{array}{lcl}
\displaystyle \frac{1}{c-a}\big(c \exp(-ax^{-\gamma})-a \exp(-cx^{-\gamma})\big), \ &~~& \ x>0,\  a\neq c,\\[0.3cm]
\big(1+cx^{-\gamma}\big)\exp(-cx^{-\gamma}), \ &~~& \ x>0,\  a=c.
\end{array}
\right.
$$
(b) If $F(x)={\rm e}^{-{a}/({\rm e}^x-1)},\ x>0$, and $G(x)={\rm e}^{- {c}/({\rm e}^x-1)},\ x>0$, with $a,c>0$, then
$$
\mathbb P(X[Y]\leq x)=G\widetilde{\#}F(x)=\left\{
\begin{array}{lcl}
\displaystyle\frac{1}{c-a}\big(c{\rm e}^{- {a}/({{\rm e}^x-1})}-a{\rm e}^{-{c}/({{\rm e}^x-1})}\big), \ &~~& \ x>0,\  a\neq c,\\[0.3cm]
\Big(1+\displaystyle\frac{c}{{\rm e}^x-1}\Big){\rm e}^{-{c}/({\rm e}^x-1)}, \ &~~& \ x>0,\  a=c.
\end{array}
\right.
$$
\end{example}
\par
Ageing notions and stochastic orders have many applications in various areas of science such as reliability and survival analysis, economics, insurance, actuarial and management sciences, and coding theory; see Shaked and Shanthikumar \cite{Shaked} for a detailed account. In the following, we review some notions that are used in the sequel. Note that here and throughout this paper, the terms `increasing' and `decreasing' are used in a nonstrict sense, and $\mathbb R$ denotes the set of real numbers.
\begin{definition}\label{def:2.2}\rm
If $X$ is an absolutely continuous random variable with support $(l_X,u_X)$, CDF $F$, PDF $f$ and reversed hazard rate function $\tau(t)=\displaystyle{f(t)}/F(t)$, then 
\begin{itemize}
  \item $X$ is said to have the increasing likelihood ratio (ILR) property if $f(x)$ is log-concave or, equivalently, the function $f'(x)/f(x)$ is decreasing in $x\in(l_X,u_X)$;
	\item $X$ is said to have the decreasing likelihood ratio (DLR) property if $f(x)$ is log-convex or, equivalently, the function $f'(x)/f(x)$ is increasing in $x\in(l_X,u_X)$;
  \item $X$ has the DRHR if $\tau(t)$ is decreasing in $t\in (l_X,u_X)$ or, equivalently, $T(x)=-\log F(x)$ is convex.
\end{itemize}
Moreover, if $Y$ is an absolutely continuous random variable with support $(l_Y,u_Y)$, CDF $G$ and PDF $g$, then 
\begin{itemize}
  \item $X$ is smaller than $Y$ in the usual stochastic order (denoted by $X\leq_{\rm st}Y$) 
  if $\bar{F}(t)\leq\bar{G}(t)$, for all $t\in \mathbb{R}$, or equivalently $F(t)\geq G(t)$, for all $t\in \mathbb{R}$; 
\item $X$ is smaller than $Y$ in the likelihood ratio order (denoted by $X\leq_{\rm lr}Y$) if $f(x)g(y)\geq f(y)g(x)$ for all  
$x\leq y$, with $x,y\in(l_X,u_X)\cup(l_Y,u_Y)$;
\item $X$ is smaller than $Y$ in the up-shifted likelihood ratio order (denoted by $X\leq_{\rm lr\uparrow}Y$) 
if $X-x\leq_{\rm lr}Y$ for all $x\geq0$ or, equivalently, for each $x\geq0$ we have $g(t)/f(t+x)$ is increasing 
in $t\in(l_X-x,u_X-x)\cup(l_Y,u_Y)$, 
where $a/0$ is taken to be equal to $\infty$ whenever $a > 0$.
\end{itemize}
\end{definition}
\begin{proposition}\label{pr:1}  
If $X$ and $Y$ are independent nonnegative random variables,  then 
$$ 
 X[Y]\leq_{\rm st}\min\{X,Y\}.
$$ 
\end{proposition}
\begin{proof}
Due to (\ref{G:symbol:F}) and (\ref{X[Y]}), for $x>0$, we have 
\begin{eqnarray*} 
P(X[Y]\leq x)&=& G(x)+F(x)\int_{x}^{\infty}\frac{1}{F(t)}{\rm d}G(t) 
\nonumber \\
&\geq & G(x)+F(x)\int_{x}^{\infty} {\rm d}G(t) 
\nonumber \\
&=& G(x)+F(x)-F(x)G(x).
\end{eqnarray*}
The proof thus follows recalling that the last term is the CDF of $\min\{X,Y\}$.  
\end{proof}
\par
Let us now consider a stochastic model that extends both cases treated in Example \ref{eq:FGab}. 
Let $X$ and $Y$ be  absolutely continuous nonnegative random variables with CDFs $F(x)$ and $G(x)$, 
and reversed hazard rate functions $\tau_X(x)$ and $\tau_Y(x)$, respectively. These variables satisfy the 
PRHRM with proportionality constant $\theta>0$ if $\tau_Y(x)=\theta\tau_X(x)$ 
for all $x>0$ or, equivalently, if 
\begin{equation}\label{reversed:H:R:Model}
G(x)=[F(x)]^\theta,\qquad  x>0.
\end{equation}
The parent distribution function can be expressed as $F(x)=e^{-T(x)},\ x>0,$ where $T(x)$ is defined in (\ref{cumulative:reversed:hazard}). The model (\ref{reversed:H:R:Model}) was first proposed by Lehman \cite{Lehman} in contrast to the proportional hazard rate model. It is more flexible to accommodate both monotonic as well as nonmonotonic failure rates even though the baseline failure rate is monotonic. For more details on the applications and properties of the proportional hazard rate model see e.g.\ \cite{Di-Crescenzo-2000},  \cite{Gupta-1998},  \cite{Gupta-1997},  
\cite{Mudholkar-Hotson} and \cite{Mudholkar-1995}, among others.
\begin{proposition}\label{prop:reversed}
Under the PRHRM (\ref{reversed:H:R:Model}), we have 
\begin{eqnarray}\label{eq:222}
G\widetilde{\#}F(x)=\frac{1}{\theta-1}\left(\theta {\rm e}^{-T(x)}-{\rm e}^{-\theta T(x)}\right),
\qquad x>0,
\end{eqnarray}
for $\theta>0$,  $\theta\neq 1$.
\end{proposition}
\begin{proof}
Under the PRHRM (\ref{reversed:H:R:Model}), we can verify that
\begin{eqnarray*} 
\int_{x}^{\infty}\frac{1}{F(t)}{\rm d}G(t)=\frac{\theta}{\theta-1}\left(1-{\rm e}^{-(\theta-1)T(x)}\right),
\qquad x>0.
\end{eqnarray*}
This identity and  (\ref{G:symbol:F}) thus yield (\ref{eq:222}).
\hfill{$\Box$}
\end{proof}
\par
Note that the assumptions of Proposition \ref{prop:reversed} are satisfied by the  cases (a) and (b) 
shown in Example \ref{eq:FGab}, for $a\neq c$.
\begin{example}\rm
Let $X_{m:m}=\max\{X_1,\ldots,X_m\}$ be the lifetime of the parallel system consisting of $m$ components 
with absolutely continuous i.i.d. lifetimes $X_1,\ldots,X_m,$ having the common cdf $F$. Moreover, 
suppose that $X_i,\ i=1,\ldots,m,$ and $Y$ satisfy the PRHRM with proportionality 
constant $\theta$, as in (\ref{reversed:H:R:Model}). Then after some calculations, from (\ref{G:symbol:F}) we can 
express the CDF of $X_{m:m}[Y]$ as the following generalized mixture: 
$$
G\widetilde{\#}F_{m:m}(x)
=\frac{\theta F^m(x) - m F^{\theta}(x)}{\theta-m},
\qquad x>0,
$$
for $\theta\neq m$.
\hfill{$\Box$}
\end{example}
\par
In the forthcoming theorem we investigate the commutative property for the reversed relevation transform.
\begin{theorem}\label{th:commut}
The reversed relevation transform of $X$ and $Y$ is commutative if and only if $X$ and $Y$ satisfy the 
PRHRM. 
\end{theorem}
\begin{proof}
If $X$ and $Y$ satisfy the PRHRM as specified in (\ref{reversed:H:R:Model}), 
then (\ref{G:symbol:F}) yields
$$
 G\widetilde{\#}F(x)=F\widetilde{\#}G(x)=\frac{\theta F(x)-G(x)}{\theta-1},\qquad  x>0,
$$
for $\theta>0$, $\theta\neq 1$, and the desired result follows. To prove the converse, we assume that, for all 
$x>0$,
\begin{equation}\label{equ:1}
G(x)+F(x)\int_{x}^{\infty}\frac{1}{F(t)}{\rm d}G(t)=F(x)+G(x)\int_{x}^{\infty}\frac{1}{G(t)}{\rm d}F(t).
\end{equation}
Differentiating both sides of \eqref{equ:1}, we have
\begin{equation}\label{equ:2}
f(x)\int_{x}^{\infty}\frac{1}{F(t)}{\rm d}G(t)=g(x)\int_{x}^{\infty}\frac{1}{G(t)}{\rm d}F(t),
\end{equation}
where $f$ and $g$ denote the PDFs of $X$ and $Y$, respectively.  
Again, differentiating both sides of \eqref{equ:2}, we obtain
\begin{equation}\label{equ:3}
f'(x)\int_{x}^{\infty}\frac{1}{F(t)}{\rm d}G(t)-\frac{f(x)g(x)}{F(x)}
 =g'(x)\int_{x}^{\infty}\frac{1}{G(t)}{\rm d}F(t)-\frac{f(x)g(x)}{G(x)}.
\end{equation}
From \eqref{equ:1}--\eqref{equ:3} and some algebraic simplification, we obtain
$$
 \frac{f'(x)}{f(x)}-\frac{f(x)}{F(x)}=\frac{g'(x)}{g(x)}-\frac{g(x)}{G(x)},\qquad  x>0,
$$
i.e.
$$
 \frac{d}{dx}\ln\frac{f(x)}{F(x)}=\frac{d}{dx}\ln\frac{g(x)}{G(x)},\qquad  x>0.
$$
Integration on both sides yields
$$
 \ln\frac{f(x)}{F(x)}=\ln\frac{g(x)}{G(x)}+\text{constant},\qquad  x>0,
$$
or
$$
 \frac{f(x)}{F(x)}=\theta\frac{g(x)}{G(x)},\qquad x>0,
$$
where $\theta$ is a positive constant. 
Thus, we obtain $G(x)=[F(x)]^{\theta}$, $x>0$, which completes the proof. 
\end{proof}
\par
Hereafter, we analyse some results that are dual to those given above. 
Let $X$ and $Y$ be absolutely continuous nonnegative random variables with survival functions 
$\bar F(x)$ and $\bar G(x)$, and hazard rate functions $h_X(x)=-({\rm d}/{\rm d}x)\log\bar{F}(x)$ 
and $h_Y(x)=-({\rm d}/{\rm d}x)\log\bar{G}(x)$, respectively. Then $X$ and $Y$ satisfy the proportional hazards rate model with proportionality constant $\theta>0$, if $h_Y(x)=\theta h_X(x)$ for all $x>0$. This is  equivalent to the model
\begin{equation}\label{H:R:Model}
\bar{G}(x)=[\bar{F}(x)]^\theta,\qquad \theta>0,
\end{equation}
where $\bar{F}(x)=e^{-\Lambda(x)},\ x>0$, is the parent survival function and $\Lambda(x)=-\log\bar{F}(x)$, denotes the hazard function; see, e.g.\ \cite{Gupta-1998}. Let $X(Y)$ denote the total time of $X$ given that it exceeds an 
independent random inspection time $Y$, i.e.\ $X(Y)=[X|X>Y]$. Then we have
\begin{eqnarray}\label{G:symbol:F:bar}
\mathbb P(X(Y)>x)&=&\bar{G}{\#}\bar{F}(x)\nonumber\\
&=&\bar{G}(x)+\bar{F}(x)\int_{0}^{x}\frac{1}{\bar{F}(t)}{\rm d}G(t),\qquad x>0,
\end{eqnarray}
where the symbol ${\#}$ denotes the relevation transform introduced by Krakowski \cite{Krakowski}. Equation (\ref{G:symbol:F:bar}) was discussed in \cite{Kapo-Psarr}; 
see also \cite{Burkschat-Navarro} and \cite{Psarrakos} and the references therein.
\par
The following result is analogous to Proposition \ref{pr:1} and, thus, the proof is omitted. 
\begin{proposition}\label{pr:1simm}  
If $X$ and $Y$ are independent nonnegative random variables,  then
$$ 
 X(Y)\geq_{\rm st}\max\{X,Y\}.
$$ 
\end{proposition}
\par
Let us now see the analogous of Proposition \ref{prop:reversed}. 
\begin{proposition}\label{prop:hazard}
Under the proportional hazards rate model (\ref{H:R:Model}), we have 
\begin{eqnarray*} 
 \bar{G}{\#}\bar{F}(x)=\frac{1}{\theta-1}\left(\theta {\rm e}^{-\Lambda(x)}-{\rm e}^{-\theta\Lambda(x)}\right),
 \qquad  x>0,
\end{eqnarray*}
for $\theta>0$, $\theta\neq 1$.
\end{proposition}
\begin{proof}
Since
\begin{eqnarray*} 
\int_{0}^{x}\frac{1}{\bar{F}(t)}{\rm d}G(t)
=\frac{\theta}{\theta-1}\left(1-{\rm e}^{-(\theta-1)\Lambda(x)}\right),
\qquad x>0,
\end{eqnarray*}
the proof follows from (\ref{G:symbol:F:bar}).
\hfill{$\Box$}
\end{proof}
\begin{example}\rm
Suppose that $X_{1:m}=\min\{X_1,\ldots,X_n\}$ is the lifetime of the series system consisting of $m$ components 
with absolutely continuous i.i.d.\ lifetimes $X_1,\ldots,X_m,$ having the common 
CDF $F$. Also, suppose that $X_i,\ i=1,\ldots,m,$ and $Y$ satisfy the proportional hazard rate model with 
proportionality constant $\theta$, as in (\ref{H:R:Model}). 
By means of  some calculations, from (\ref{G:symbol:F:bar}) we obtain the survival function of $X_{1:m}(Y)$, 
which is expressed as a generalized mixture: 
$$
\bar{G}{\#}\bar{F}_{1:m}(x)=\frac{\theta \bar{F}^m(x)-m\bar{F}^{\theta}(x)}{\theta-m},
\qquad x>0,
$$
for $\theta\neq m$. 
\hfill{$\Box$}
\end{example}
\par
The relevation transform is not always commutative. Indeed, in the following theorem 
we give a necessary and sufficient condition leading to such a property. Being similar 
to Theorem \ref{th:commut} we provide only a sketch of the proof.  
\begin{theorem}\label{th:commut2}
The relevation transform of $X$ and $Y$ is commutative if and only if $X$ and $Y$ satisfy the 
proportional  hazard rate model. 
\end{theorem}
\begin{proof}
Let $X$ and $Y$ satisfy the proportional hazard rate model as in (\ref{H:R:Model}).  
From (\ref{G:symbol:F:bar}), we thus have
$$
  \bar{G}{\#}\bar{F}(x)
  = \bar F{\#} \bar G(x)=\frac{\theta \bar F(x)-\bar G(x)}{\theta-1},
  \qquad  x>0,
$$
for $\theta>0$, $\theta\neq 1$, and then the relevation transform is commutative. 
To prove the converse, we assume that for all $x>0$
$$
 \bar{G}(x)+\bar{F}(x)\int_{0}^{x}\frac{1}{\bar{F}(t)}{\rm d}G(t)
 = \bar{F}(x)+\bar{G}(x)\int_{0}^{x}\frac{1}{\bar{G}(t)}{\rm d}F(t). 
$$
Differentiating both sides and after some calculations, we obtain
$$
 \frac{f'(x)}{f(x)}+\frac{f(x)}{\bar F(x)}=\frac{g'(x)}{g(x)}+\frac{g(x)}{\bar G(x)},\qquad  x>0,
$$
so that 
$$
 \frac{\rm d}{{\rm d}x}\ln\frac{f(x)}{\bar F(x)}=\frac{\rm d}{{\rm d}x}\ln\frac{g(x)}{\bar G(x)},\qquad  x>0.
$$
Such a relation implies that $\bar G(x)=[\bar F(x)]^{\theta}$, $x>0$, for $\theta>0$, thus completing the proof. 
\end{proof}
\par
The characterization of distributions based on the relevation transform has been the object of various investigations;   
see, e.g.\ Lau and Prakasa Rao \cite{LauPRao1}  and \cite{LauPRao2}. We point out that 
[14, Theorem 9] states that the relevation transform of two i.i.d.\ nonnegative continuous random variables is identically distributed to their convolution, 
i.e.\ $F\# F(x)=F*F(x)$ for all $x\geq 0$, if and only if they have exponential distribution. However, due to 
Proposition \ref{pr:1}, a similar result cannot hold for the reversed relevation transform. 
\section{Sequence of weighted distributions}
Let $X$ be an absolutely continuous nonnegative random variable with PDF $f(x)$ and CDF $F(x)$. 
Based on $X$, we construct a sequence of random variables $\{X_n,n\geq1\}$ as  
\begin{equation} 
X_1\stackrel{D}{=}X,  
\qquad 
[X_{n+1}\,|\,X_n=t]\stackrel{D}{=}[X_n\,|\,X_n\leq t],\qquad n\geq 1,
\label{eq:relXn}
\end{equation}
or equivalently
\begin{equation}
 X_1\stackrel{D}{=}X, \qquad 
 X_{n+1}\stackrel{D}{=}[X_n\,|\,X_n\leq X_n'],\qquad n\geq 1,
\label{sequence}
\end{equation}
where $X'_n$ is an independent copy of $X_n$. It is easy to show that the corresponding CDFs 
$F_n(x)=\mathbb P(X_n\leq x)$ are given as, for all $x>0$,
\begin{equation}\label{cdf:F:n+1}
 F_1(x)=F(x),\qquad 
 F_{n+1}(x)=F_n\widetilde{\#}F_n(x)=F_n(x)[1+T_n(x)],\qquad n\geq 1,
\end{equation}
where
\begin{equation}\label{cumulative:hazard:function}
T_n(x)=-\log F_n(x)=\int_{x}^{\infty}\tau_n(u){\rm d}u,\qquad x>0,
\end{equation}
denotes the cumulative hazard function of $X_n$, and
\begin{equation}\label{reversed:hazard:n}
\tau_n(u)=\frac{f_n(u)}{F_n(u)},\qquad u>0,
\end{equation}
is the reversed hazard rate of $X_n$. From   (\ref{cdf:F:n+1}), 
we can see that the corresponding densities are given by, for $x>0$,
\begin{equation}
 f_1(x)=f(x),\qquad 
 f_{n+1}(x)=T_n(x)f_n(x),\qquad n\geq 1.\label{pdf:F:n+1}
\end{equation}
Due to (\ref{cumulative:hazard:function}), for $x>0$, we have  
\begin{eqnarray*} 
T_{n+1}(x)=-\log F_{n+1}(x)=T_n(x)-\log(1+T_n(x)),\qquad n\geq 1,
\end{eqnarray*}
and, thus, from (\ref{pdf:F:n+1}), we obtain 
\begin{eqnarray}\label{density:n+1:2}
f_{n+1}(x)=T_n(x)f_n(x)=T_n(x)T_{n-1}(x)f_{n-1}(x)=\cdots=\prod_{i=1}^{n}T_i(x)f(x),\qquad n\geq 1.
\end{eqnarray}
From (\ref{density:n+1:2}), we see that $f_{n+1}(x)$ is a sequence of weighted PDFs. 
We recall that, given an absolutely continuous random variable $X$ having density $f$ 
and a nonnegative real function $w$, the associated weighted random variable $X^w$ has the PDF 
$$
 f^w(x)=\frac{w(x)f(x)}{\mathbb E[w(X)]}, \qquad x\in \mathbb{R},
$$
provided that $0<\mathbb E[w(X)]<\infty$. See  \cite{LiYuHu12},    \cite{NaSuLi11}, and   \cite{Bart09}
for some recent papers on weighted distributions. 
\par
The sequence of random variables $\{X_n,n\geq1\}$ is suitable to describe an iterative process, where 
$X_n$ denotes the random time required to perform a task at the $n$th stage. 
For instance, consider a training procedure where, given that the $n$th learning time $X_n$ has duration $t$, 
the $(n+1)$th random time is identically distributed to $X_n$ conditional on $X_n\leq t$. 
In some sense,   (\ref{eq:relXn}) expresses that the information collected at each stage allows the 
next step of the procedure to have a stochastically smaller duration. Alternatively, $\{X_n,n\geq1\}$ may be viewed as 
the sequence of lifetimes of an item that is repaired instantaneously after each failure, such that after each 
repair the duration of the next lifetime is stochastically smaller than the previous, due to imperfect repairs and 
weakening caused by wear. 
\par
From   (\ref{cdf:F:n+1}), since $T_n(x)\geq0$, for all $x>0$ and $n\geq1$, we derive that $F_{n+1}(x)\geq F_{n}(x)$ for 
all $x>0$. Hence, we conclude that $X_{n}\geq_{\rm st}X_{n+1}$ for all $n\geq1$. In the following theorem, we obtain the same result for a stronger stochastic order.
\begin{theorem}\label{thm:LR:order}
Consider the sequence of random variables $\{X_n,n\geq1\}$ as defined in (\ref{sequence}). For all $n=1,2,\cdots$, we have
\begin{eqnarray*} 
X_n\geq_{\rm lr}X_{n+1}.
\end{eqnarray*}
\end{theorem}
\begin{proof}
From the definition of likelihood ratio order and (\ref{pdf:F:n+1}), we conclude that
\begin{eqnarray}\label{equation:1}
\frac{f_{n+1}(x)}{f_n(x)}=T_n(x),\qquad n\geq1.
\end{eqnarray}
The right-hand-side of (\ref{equation:1}) is decreasing in $x>0$ and, hence, the claimed result follows.
\end{proof}
\begin{theorem} 
Let $n\geq1$. If $X_n$ is ILR then
\begin{eqnarray*} 
X_n\geq_{\rm lr\uparrow}X_{n+1}.
\end{eqnarray*}
\end{theorem}
\begin{proof}
From the definition of up-shifted likelihood ratio, it is sufficient to prove that the function $f_{n+1}(x+t)/f_n(x)$ is decreasing in $x$ for all $t>0$. First, from (\ref{pdf:F:n+1}), we observe that
\begin{eqnarray*} 
\frac{f_{n+1}(x+t)}{f_n(x)}=T_n(x+t)\frac{f_{n}(x+t)}{f_n(x)}.
\end{eqnarray*}
On the other hand, recalling (\ref{cumulative:hazard:function}), we have
\begin{eqnarray}
&& \frac{\partial}{\partial x}\left\{T_n(x+t)\frac{f_{n}(x+t)}{f_n(x)}\right\} 
\nonumber\\
&&\quad =\frac{[-\tau_n(x+t)f_n(x+t)+T_n(x+t)f'_n(x+t)]f_n(x)-T_{n}(x+t)f_n(x+t)f'_n(x)}{f^2_n(x)}
\nonumber\\
&&\quad \leq \frac{T_n(x+t)[f'_n(x+t)f_n(x)-f_n(x+t)f'_n(x)]}{f^2_n(x)}.
\label{differntiate} 
\end{eqnarray}
The last Expression in (\ref{differntiate}) is negative since, by  assumption, $X_n$ is ILR, i.e.
\begin{eqnarray*} 
\frac{f'_{n}(x)}{f_n(x)}\geq\frac{f'_{n}(x+t)}{f_n(x+t)},\qquad \hbox{for all\ } 0<x\leq t+x.
\end{eqnarray*}
The proof is thus completed.
\end{proof}
\begin{remark}\label{rem:DRHR}\rm 
If $X$ is DRHR, i.e.\ $\tau(x)$ is decreasing, then $T(x)$ is convex. 
\end{remark}
\begin{theorem} 
Let $n\geq1$. If $X_n$ is DLR then $X_{n+1}$ is DLR.
\end{theorem}
\begin{proof}
From (\ref{pdf:F:n+1}), we need to show that
\begin{eqnarray}\label{thm:proof:DLR}
\log f_{n+1}(x)=\log f_{n}(x)+\log T_n(x)
\end{eqnarray}
is convex for all $x>0$. If $X_n$ is DLR then $\log f_{n}(x)$ is convex. 
On the other hand, it is well known that if $X_n$ is DLR then $X_n$ is DRHR and, hence, $T_n(x)$ is convex due to Remark \ref{rem:DRHR}. Since the function $\log(\cdot)$ is convex and increasing,   $\log T_n(x)$ is convex. From (\ref{thm:proof:DLR}), we have that $\log f_{n+1}(x)$ is the sum of two convex functions and then the desired result follows.
\hfill{$\Box$}
\end{proof}
\begin{remark}\rm 
If $X_n$ is ILR then $X_{n+1}$ is not necessarily ILR. To show this fact, consider the following example.
\end{remark}
\begin{example}\rm
Suppose that $X_1$ has CDF $F(x)=x^{\alpha},\ 0<x<1\ (\alpha>0)$. It is easy to see that
$\tau(x)=\alpha/x$, and $T(x)=-\alpha\log x,\ 0<x<1$. It follows that
\begin{eqnarray*} 
\frac{f'_1(x)}{f_1(x)}=\frac{\alpha-1}{x},\qquad 0<x<1,
\end{eqnarray*}
is decreasing in $x\in(0,1)$ for all $\alpha\geq1$ and, hence, $X_1$ is ILR. 
(Note that $T(x)$ is convex according to Remark \ref{rem:DRHR}.)
On the other hand, we have
\begin{eqnarray}\label{D:second}
\frac{f'_2(x)}{f_2(x)}=\frac{1+(\alpha-1)\log x}{x\log x},
\qquad 0<x<1.
\end{eqnarray}
It is easy to show that the right-hand side of (\ref{D:second}) is not decreasing in $x$ for all $\alpha\geq1$ and, thus, $X_2$ is not ILR.
\end{example}
\begin{proposition}\label{DRHR:q:x}
Let $q(x)$ be a nonnegative function of $x>0$. If $q(x)\tau_1(x)$ is a decreasing function of $x>0$, then $q(x)\tau_n(x)$ is also a decreasing function of $x>0$ for all $n=1,2,\ldots$.
\end{proposition}
\begin{proof}
We just show that under the hypothesis, the function $q(x)\tau_2(x)$ is a decreasing function of $x>0$. From   
(\ref{cdf:F:n+1}) and (\ref{pdf:F:n+1}), we have
\begin{equation}\label{thm:proof:q:X:n}
q(x)\tau_2(x)=q(x)\frac{f_2(x)}{F_2(x)} 
=q(x)\frac{T_1(x)f_1(x)}{(1+T_1(x))F_1(x)} 
=q(x)\tau_1(x)\frac{T_1(x)}{1+T_1(x)}.
\end{equation}
Since $x/(x+1)$ is an increasing function of $x>0$, and the function $T_1(x)$ is decreasing with respect to $x>0$, then the function ${T_1(x)}/({1+T_1(x)})$ is decreasing with respect to $x>0$. From (\ref{thm:proof:q:X:n}), we thus obtain that $q(x)\tau_2(x)$ is a decreasing function of $x>0$. The rest of the proof follows by induction.
\end{proof}
\begin{corollary}
If $X_1$ is DRHR then $X_n$ is DRHR for all $n\geq2$.
\end{corollary}
\begin{proof}
The proof follows from Proposition \ref{DRHR:q:x} by taking $q(x)=1$ for all $x>0$.
\end{proof}
\par
Let us now consider the following property. 
\begin{definition}\label{def:LB-DRHR}\rm
Let $X$ be an absolutely continuous random variable with support $(l_X,u_X)$. 
We say that $X$ has the DRHR in a length-biased sense (LBDRHR) if 
$x\,\tau(x)$ is decreasing in $x\in (l_X,u_X)$. 
\end{definition}
\par
We remark that a necessary and sufficient condition such that $X$ is LBDRHR has been given 
in terms of stochastic comparison of quantile-based distributions in  \cite{DiCMaMu}. 
Other results on the characterization given in Definition \ref{def:LB-DRHR} will be the object of 
a future investigation. 
\begin{corollary}
If $X_1$ is LBDRHR, then $X_n$ is LBDRHR for all $n\geq2$.
\end{corollary}
\begin{proof}
The proof follows from Proposition \ref{DRHR:q:x} by taking $q(x)=x$ for all $x>0$.
\hfill{$\Box$}
\end{proof}
\par
Consider now the following stochastic order from \cite{Rezaei}. 
\begin{definition} \rm 
Let $X$ and $Y$ be absolutely continuous random variables with reversed hazard rates $\tau_X(x)$ and 
$\tau_Y(x)$, respectively. The random variable $X$ is said to be smaller than $Y$ in relative reversed hazard rate order (denoted by $X\leq_{\rm RRH}Y$), if $\tau_Y(x)/\tau_X(x)$ is an increasing function of $x$.
\end{definition}
\par
For instance, let $X$ and $Y$ denote the lifetimes of two components; given that the components have been found to be failed at the same time, then $X\leq_{\rm RRH}Y$ states that $Y$ has been lived longer than $X$ or, equivalently, $X$ aged faster than $Y$.
\begin{proposition}\label{prop:RRH}
The sequence of random variables defined in (\ref{sequence}) satisfies $X_{n}\geq_{\rm RRH}X_{n+1}$ for all $n\geq1$.
\end{proposition}
\begin{proof}
From (\ref{cdf:F:n+1}), (\ref{reversed:hazard:n}) and (\ref{pdf:F:n+1}), we have
$$
\frac{\tau_{n+1}(x)}{\tau_{n}(x)}=\frac{f_{n+1}(x)F_n(x)}{F_{n+1}(x)f_n(x)}=\frac{T_n(x)}{1+T_n(x)}.
$$
We can see that the function $ {T_n(x)}/({1+T_n(x)})$ is decreasing with respect to $x>0$, for all $n\geq1$, which completes the proof.
\hfill{$\Box$}
\end{proof}
\par
We remark that the results stated in Theorem \ref{thm:LR:order} and Proposition \ref{prop:RRH} are not related to each other, since   likelihood ratio order does not imply $RRH$ rate order and vice versa.
\begin{theorem}\label{AFR:Y}
Consider the sequence of random variables $\{X_n,n\geq1\}$ as defined in (\ref{sequence}) with the reversed hazard rate functions defined in (\ref{reversed:hazard:n}), and let $Y$ be an absolutely continuous nonnegative 
random variable with the reversed hazard rate function $\tau_Y(x),\ x>0$.
If $X_1$ and $Y$ satisfy the PRHRM with $G(x)=[F_1(x)]^\theta,\ x>0$, then $X_n\leq_{\rm RRH}Y$ for all $n=1,2,\ldots$.
\end{theorem}
\begin{proof}
Since $\tau_1(x)/\tau_Y(x)=\theta^{-1},\ x>0$, from (\ref{reversed:H:R:Model}) we have that
$$
\frac{\tau_2(x)}{\tau_Y(x)}=\frac{f_2(x)}{F_2(x)\tau_Y(x)} 
=\frac{\tau_1(x)}{\tau_Y(x)}\frac{T_1(x)}{1+T_1(x)} 
=\theta^{-1}\frac{T_1(x)}{1+T_1(x)}.
$$
The function ${T_1(x)}/({1+T_1(x)})$ is a decreasing with respect to $x>0$ and, hence, $X_2\leq_{\rm RRH}Y$. 
The rest of the proof follows by induction.
\hfill{$\Box$}
\end{proof}
\section{Connection with entropy and covariance}
In this section we obtain some results about the connection between the entropy and the inactivity time of the new weighted distribution function considered in the previous section. One of the most important measure of uncertainty is the differential entropy introduced by Shannon \cite{Shannon}. 
For an  absolutely continuous nonnegative random variable $X$ having PDF $f$,  
the differential entropy is defined by
\begin{eqnarray}\label{entropy}
H(X)=-\int_{0}^{\infty}{f}(x)\log{f}(x){\rm d}x,
\end{eqnarray}
where `$\log$' means natural logarithm and, by convention, $0\log 0=0$. 
The entropy $H(X)$ gives expected uncertainty contained in $f(t)$ about the predictability of an outcome of the random variable $X$. It is known that in many realistic situations, such as in survival analysis and reliability, one has information about the past time, i.e.\ the time elapsed after failure till time $t$, given that the unit has already failed. The entropy (\ref{entropy}) applied to conditioned random variable is useful to measure uncertainty in such situations. Di Crescenzo and Longobardi \cite{Di-Longobardi-2002} indeed considered the entropy for the past lifetime, called past entropy at time $t$ of $X$, denoted by
$$
\bar{H}(t)=-\int_{0}^{t}\frac{f(x)}{F(t)}\log\frac{f(x)}{F(t)}{\rm d}x,\qquad t>0;
$$
see also  \cite{Muliere}. Furthermore, the concept of dynamic cumulative entropy as an alternative measure of uncertainty for the inactivity time was introduced in \cite{Di-Longobardi-2009} and is defined as
\begin{eqnarray}\label{past:ce}
\mathcal{CE}(X;t)&=&-\int_{0}^{t}\frac{F(x)}{F(t)}\log\frac{F(x)}{F(t)}{\rm d}x\nonumber\\
&=&-\frac{1}{F(t)}\int_{0}^{t}F(x)\log F(x){\rm d}x-T(t)\tilde{\mu}(t),\qquad t>0,
\end{eqnarray}
where $\tilde{\mu}(\cdot)$ and $T(\cdot)$ are defined in   (\ref{def:mean:residual}) and (\ref{cumulative:reversed:hazard}), respectively. Note that
\begin{eqnarray}\label{ce}
\mathcal{CE}(X)=\lim_{t\rightarrow\infty}\mathcal{CE}(X;t)
=-\int_{0}^{\infty}F(x)\log{F}(x){\rm d}x
=\int_{0}^{\infty}F(x)T(x){\rm d}x,
\end{eqnarray}
where $\mathcal{CE}(X)$ is called cumulative entropy of $X$. Di Crescenzo and Longobardi 
\cite{Di-Longobardi-2009} also showed that the dynamic cumulative entropy and the mean inactivity time are connected as 
\begin{eqnarray*} 
\mathcal{CE}(X;t)=\mathbb E[\tilde{\mu}(X)\,|\,X\leq t],\qquad t>0
\end{eqnarray*}
and, thus,
\begin{eqnarray}\label{past:ce:mean:inactivity}
\mathcal{CE}(X)=\mathbb E[\tilde{\mu}(X)].
\end{eqnarray}
Now, we extend the results of  \cite{Kapo-Psarr} to the case of past time. 
Based on $X_1\stackrel{D}{=}X$, we consider the sequence of random variables $\{X_n,n\geq1\}$, with the corresponding distributions $F_n(x)$ as defined in (\ref{cdf:F:n+1}), and denote by 
\begin{equation}\label{mean:past:time}
\mu_n(t)=E[X_n\,|\,X_n\leq t],\qquad t>0,
\end{equation}
the mean past lifetime of $X_n,\ n\geq1$. Hereafter, we obtain the main results and connections between the dynamic cumulative entropy and the reversed hazard rate function.
\begin{theorem}\label{thm:condition:covariance}
For any $t>0$, and for all $n=1,2,\ldots$, we have
\begin{eqnarray*} 
\mathbb E\left[\frac{T_n(X_n)}{\tau_n(X_n)}\,\Big|\,X_n\leq t\right]=\mathcal{CE}(X_n;t)+\tilde{\mu}_n(t)T_n(t),
\end{eqnarray*}
\begin{eqnarray}\label{Cov:X:TX}
{\rm cov}\left[X_n,T_n(X_n)\,|\,X_n\leq t\right]=T_n(t)\left[\mu_n(t)-E(X_n)\right]-\mathcal{CE}(X_n;t).
\end{eqnarray}
\end{theorem}
\begin{proof}
For $t>0$, and $n\geq1$, we have
\begin{eqnarray*} 
\mathbb E\left[\frac{T_n(X_n)}{\tau_n(X_n)}\Big|X_n\leq t\right]&=&
\frac{1}{F_n(t)}\int_{0}^{t}\frac{T_n(x)}{\tau_n(x)}f_n(x){\rm d}x\nonumber\\
&=&\frac{1}{F_n(t)}\int_{0}^{t}T_n(x)F_n(x){\rm d}x\nonumber\\
&=&\mathcal{CE}(X_n;t)+\tilde{\mu}_n(t)T_n(t),
\end{eqnarray*}
where the last equality is obtained from (\ref{past:ce}). To prove the second expression, it is easy to show that the random variable $T_n(X_n)$ is exponentially distributed with unity mean and, hence, $\mathbb E[T_n(X_n)]=1$. Now, consider the following expression for $t>0$:
\begin{eqnarray}\label{thm:1}
{\rm cov}\left[X_n,T_n(X_n)\,|\,X_n\leq t\right]&=&
\mathbb E\left\{[X_n-\mathbb E(X_n)][T_n(X_n)-\mathbb E(T_n(X_n))]\,|\,X_n\leq t\right\}
\nonumber\\
&=&\mathbb E[X_nT_n(X_n)\,|\,X_n\leq t]-\mu_n(t) \nonumber\\
&& -\mathbb E(X_n)\mathbb E[T_n(X_n)\,|\,X_n\leq t]+\mathbb E(X_n).
\end{eqnarray}
We have that
\begin{eqnarray}\label{thm:2}
\mathbb E[T_n(X_n)\,|\,X_n\leq t]&=&\frac{1}{F_n(t)}\int_{0}^{t}T_n(x)f_n(x){\rm d}x \\
&=&\frac{1}{F_n(t)}\int_{0}^{t}f_{n+1}(x){\rm d}x\nonumber\\
&=&\frac{F_{n+1}(t)}{F_n(t)}\nonumber\\
&=&1+T_n(t).
\end{eqnarray}
Also, we see that
\begin{eqnarray*}\label{thm:3}
\mathbb E[X_nT_n(X_n)\,|\,X_n\leq t]=\frac{1}{F_n(t)}\int_{0}^{t}xf_{n+1}(x){\rm d}x.
\end{eqnarray*}
Integrating by parts, we can derive
\begin{eqnarray*}\label{thm:4}
\mathbb E[X_nT_n(X_n)\,|\,X_n\leq t]&=&\frac{1}{F_n(t)}\left[tF_{n+1}(t)-\int_{0}^{t}F_{n+1}(x){\rm d}x\right]\\
&=&t[1+T_n(t)]-\frac{1}{F_n(t)}\int_{0}^{t}[1+T_n(x)]F_n(x){\rm d}x,
\end{eqnarray*}
and by using (\ref{past:ce}) and after simplification, we obtain
\begin{eqnarray}\label{thm:5}
\mathbb E[X_nT_n(X_n)\,|\,X_n\leq t]=\mu_n(t)[1+T_n(t)]-\mathcal{CE}(X_n;t).
\end{eqnarray}
Substituting (\ref{thm:2}) and (\ref{thm:5}) into (\ref{thm:1}), the desired result (\ref{Cov:X:TX}) finally follows.
\hfill{$\Box$}
\end{proof}
\par
From Theorem \ref{thm:condition:covariance}, the following corollary is derived.
\begin{corollary}
Under the conditions of Theorem \ref{thm:condition:covariance}, we have, for $n\geq1$,
$$
  \displaystyle\lim_{t\rightarrow\infty}\mathbb E\left[\frac{T_n(X_n)}{\tau_n(X_n)}\,\Big|\,X_n\leq t\right]
  =\mathbb E  \left[\frac{T_n(X_n)}{\tau_n(X_n)}\right]=\mathcal{CE}(X_n),
$$
$$
 \displaystyle\lim_{t\rightarrow\infty}{\rm cov}\left[X_n,T_n(X_n)\,|\,X_n\leq t\right]={\rm cov}\left(X_n,T_n(X_n)\right)=-\mathcal{CE}(X_n).
$$
\end{corollary}
\begin{remark}\rm 
Note that the initial random variable was arbitrary selected. Therefore, for any absolutely continuous nonnegative random variable $X$, the following identities hold:
$$
 \mathbb E\left(\frac{T(X)}{\tau(X)}\right)=\mathcal{CE}(X), \qquad 
 {\rm cov}(X,T(X))=-\mathcal{CE}(X).
$$
Moreover, for all $t>0$, we have
$$
 \mathbb E\left[\frac{T(X)}{\tau(X)}\,\Big|\,X\leq t\right]=\mathcal{CE}(X;t)+\tilde{\mu}(t)T(t),
$$
$$
 {\rm cov}[X,T(X)\,|\,X\leq t]=T(t)[\mu(t)-\mathbb E(X)]-\mathcal{CE}(X;t),
$$
where $\mu(t)=\mathbb E[X\,|\,X\leq t],\ t>0,$ denotes the mean past lifetime of $X$.
\end{remark}
\begin{remark}\label{rem:5}\rm
Consider the sequence of random variables as defined in (\ref{sequence}). From (\ref{cdf:F:n+1}), we have, for $n\geq1$,
$$
 \bar{F}_{n+1}(x)=1-F_{n+1}(x)=\bar{F}_n(x)-T_n(x)F_n(x),\qquad x>0.
$$
Hence, recalling (\ref{ce}), we obtain the following iterative expression for the mean of $X_n$:
\begin{eqnarray}\label{E:X:n+1}
\mathbb E(X_{n+1})= \mathbb E(X_n)-\mathcal{CE}(X_n),\qquad n\geq1,
\end{eqnarray}
which also gives a new probabilistic meaning for the cumulative entropy. 
\end{remark}
\begin{theorem} 
For all $n=1,2,\ldots$, we have
\begin{eqnarray}\label{cov:Xn:Xn+1}
{\rm cov}(X_n,X_{n+1})={\rm cov}(X_n,X_n-\tilde{\mu}(X_n))={\rm var}(X_n)-{\rm cov}(X_n,\tilde{\mu}(X_n)).
\end{eqnarray}
\end{theorem}
\begin{proof}
From (\ref{def:mean:residual}) and (\ref{mean:past:time}), we have $\mu(t)=t-\tilde{\mu}(t)$, and then
\begin{eqnarray*} 
\mathbb E\left[X_nX_{n+1}|X_n=t\right]&=&t \mathbb E\left[X_{n+1}|X_n=t\right]\nonumber\\
&=&t \mathbb E\left[X_{n}|X_n\leq t\right]\nonumber\\
&=&t\{t-\tilde{\mu}_n(t)\},
\end{eqnarray*}
for $t>0$. Hence,
\begin{eqnarray}\label{eqaua:1}
\mathbb E\left[X_nX_{n+1}\right]&=&\mathbb E[X^2_n]- \mathbb E[X_n\tilde{\mu}_n(X_n)].
\end{eqnarray}
Moreover from (\ref{E:X:n+1}) and (\ref{eqaua:1}), we have
\begin{eqnarray*} 
{\rm cov}\left(X_n,X_{n+1}\right)&=&\mathbb E[X^2_n]- \mathbb E[X_n\tilde{\mu}_n(X_n)]-\mathbb E^2[X_n]+\mathbb E[X_n]\mathcal{CE}(X_n)\nonumber\\
&=&{\rm var}[X_n]-{\rm cov}(X_n,\tilde{\mu}_n(X_n)) \nonumber\\
&=&{\rm cov}(X_n,X_n-\tilde{\mu}_n(X_n)).
\end{eqnarray*}
The second equality follows from $\mathcal{CE}(X_n)= \mathbb E[\tilde{\mu}_n(X_n)]$, due to (\ref{past:ce:mean:inactivity}). The desired result then follows.
\hfill{$\Box$}
\end{proof}
\par
Now we use the probabilistic mean value theorem (see  \cite{Di-Crescenzo-1999}) to obtain an iterative result 
for $\mathcal{CE}(X_{n+1})$. We first recall the following result (see [12, Equation (12)]). 
\begin{lemma}
The derivative of the mean inactivity time of $X$, given in  (\ref{def:mean:residual}), can be expressed   
in terms of the reversed hazard rate function (when existing) as 
\begin{equation}
 \tilde \mu'(t)= 1-\tau(t)  \tilde \mu(t), \qquad t>0: \;\; F(t)>0. 
 \label{eq:derivtildemu}
\end{equation}
\end{lemma}
\begin{theorem} 
For the sequence of random variables $\{X_n,n\geq1\}$ defined in (\ref{sequence}) and for all $n=1,2,\ldots$, we have
\begin{eqnarray}\label{thm:Xn+1}
\mathcal{CE}(X_{n+1})=\mathbb E[\tilde{\mu}_{n+1}(X_n)]-\mathbb E[\tilde{\mu}_{n+1}'(Z)]\mathcal{CE}(X_n),
\end{eqnarray}
where $\tilde{\mu}_{n+1}'(t)$ can be obtained from (\ref{eq:derivtildemu}) and $Z$ is an 
absolutely continuous nonnegative random variable having PDF
\begin{eqnarray}\label{Z:PDF}
f_{Z}(z)=\frac{F_{n+1}(z)-F_n(z)}{\mathbb E(X_n)-\mathbb E(X_{n+1})}=\frac{F_n(z)T_n(z)}{\mathcal{CE}(X_n)},
\qquad z>0.
\end{eqnarray}
\end{theorem}
\begin{proof}
Since $X_{n+1}\leq_{\rm st}X_n$ and $\mathcal{CE}(X_{n+1})=\mathbb E[\tilde{\mu}_{n+1}(X_{n+1})]$, the desired result immediately follows from [9, Proposition 3.1 and Theorem 4.1]. Note that the pdf in (\ref{Z:PDF}) is obtained from (\ref{cdf:F:n+1}) and (\ref{E:X:n+1}).
\hfill{$\Box$}
\end{proof}
For the sequence of random variables  defined in (\ref{sequence}) one can verify that, for $t>0$,
\begin{eqnarray}\label{derivate:mu}
\tilde{\mu}_{n+1}(t)=\tilde{\mu}_{n}(t)+\frac{\mathcal{CE}(X_n;t)}{1+T_n(t)}, \qquad n=1,2,\ldots.
\end{eqnarray}
\noindent Hence, from (\ref{derivate:mu}), we can write  (\ref{thm:Xn+1}) as 
\begin{eqnarray*} 
\mathcal{CE}(X_{n+1})=\mathcal{CE}(X_n)\left(1-\mathbb E[\tilde{\mu}_{n+1}'(Z)]\right)+\mathbb E\left[\frac{\mathcal{CE}(X_n;X_n)}{1+T_n(X_n)}\right], \qquad n=1,2,\ldots.
\end{eqnarray*}
%
\section{An integral operator}
Stimulated by some results shown in [18, Section 4],  
we now define a new operator which is dual to the $\textbf{T}_s$ operator introduced by 
Dickson and Hipp \cite{Dickson-Hipp}. It is known that the Dickson--Hipp operator for any $s\in \mathbb R$, 
denoted by $\widetilde{\textbf{T}}_sf(t)$, is defined by 
$$
 \textbf{T}_sf(t)=\int_{t}^{\infty}e^{-s(x-t)}f(x){\rm d}x,\qquad t>0,
$$
where $f(x)$ is an integrable function. Applications and properties of $\textbf{T}_s$ operator can be found in 
 \cite{Cai},  \cite{Dickson-Hipp}, and  \cite{Li-Garrido}, among others. For simplicity, we define an 
operator $\widetilde{\textbf{T}}_sf(x)$ for an integrable function $f$ and for $s\in \mathbb{R}$,   by
\begin{equation}\label{equation:T:f}
 \widetilde{\textbf{T}}_sf(t)=\int_{0}^{t}e^{-s(t-x)}f(x){\rm d}x,
 \qquad t>0.
\end{equation}
We can see that, for $t>0$, the two operators are related by the following identity:
$$
 \widetilde{\textbf{T}}_{-s}f(t)+ {\textbf{T}}_sf(t)=e^{st}\mathcal{L}_s[f],\qquad s\geq0,
$$
where
$$
 \mathcal{L}_s[f]=\int_{0}^{\infty}e^{-sx}f(x){\rm d}x,\qquad s\geq0,
$$
denotes the Laplace transform of the function $f$. Note that   (\ref{equation:T:f}) recalls a 
generalized Hardy operator, similar to that considered in [5, Definition 2]. Moreover, 
$\textbf{T}_sf(t)$ can be viewed as a convolution-type operator; see \cite{AbWh96}. 
For instance, if $s>0$  and $f(x)$ is a PDF then, due to (\ref{equation:T:f}), 
$(1/s) \widetilde{\textbf{T}}_sf(t)$ 
is the convolution between $f(x)$ and an exponential PDF with parameter $s$. 
\par
Suppose that $X$ is an absolutely continuous nonnegative random variable with  the CDF $F$. 
Then the $\widetilde{\textbf{T}}_s$ operator of $F$ is defined by
\begin{equation}\label{Tilde:operator}
\widetilde{\textbf{T}}_sF(t)=\int_{0}^{t}e^{-s(t-x)}F(x){\rm d}x,\qquad t>0, \  s\in \mathbb{R}.
\end{equation}
Integrating by parts, from (\ref{equation:T:f}) and (\ref{Tilde:operator}), it follows that
$$
\widetilde{\textbf{T}}_sf(t)=F(t)-s\widetilde{\textbf{T}}_sF(t),\qquad t>0.
$$
If $X$ is  an arbitrary absolutely continuous random variable with the pdf $f$ 
then, from (\ref{equation:T:f}), we have
$$
\lim_{t\rightarrow\infty}e^{st}\widetilde{\textbf{T}}_sf(t)=\mathbb E[{\rm e}^{sX}].
$$
Hence, (\ref{equation:T:f}) can be written in terms of the moment generating function of $X$ when $t$ goes to infinity. Now, we have an iterative result for $\widetilde{\textbf{T}}_sF_{n}(t)$.
\begin{theorem}\label{thm:T:tilda:q:n:s}
Let $\{X_n,n\geq1\}$ be a sequence of absolutely continuous nonnegative random variables with the corresponding CDFs defined in (\ref{cdf:F:n+1}) and let $\widetilde{\textbf{T}}_sF_n(t)$ be defined as in (\ref{Tilde:operator}). Then, for $t>0$,
\begin{equation}\label{T:tilda:q:n:s}
\widetilde{\textbf{T}}_sF_{n+1}(t)=[1+T_n(t)]\widetilde{\textbf{T}}_sF_n(t)+\widetilde{\textbf{T}}_sq_n(s,t),
\qquad n\geq1,\ s\in \mathbb{R},
\end{equation}
where $q_n(s,t)=\tau_n(t)\widetilde{\textbf{T}}_sF_n(t)$.
\end{theorem}
\begin{proof}
From  (\ref{cdf:F:n+1}) and (\ref{Tilde:operator}), we have
\begin{eqnarray*}
\widetilde{\textbf{T}}_sF_{n+1}(t)&=&\int_{0}^{t}e^{-s(t-x)}F_n(x){\rm d}x+\int_{0}^{t}e^{-s(t-x)}T_n(x)F_n(x){\rm d}x\\
&=&\widetilde{\textbf{T}}_sF_n(t)+\int_{0}^{t}e^{-s(t-x)}\left[\int_{x}^{\infty}\tau_n(u){\rm d}u\right]F_n(x){\rm d}x\\
&=&\widetilde{\textbf{T}}_sF_n(t)+\int_{0}^{t}e^{-s(t-x)}\left[\int_{x}^{t}\tau_n(u){\rm d}u\right]F_n(x){\rm d}x\\
&&+\int_{0}^{t}e^{-s(t-x)}\left[\int_{t}^{\infty}\tau_n(u){\rm d} u\right]F_n(x){\rm d}x\\
&=&[1+T_n(t)]\widetilde{\textbf{T}}_sF_n(t)+\int_{0}^{t}e^{-s(t-u)}\tau_n(u)\left[\int_{0}^{u}e^{-s(u-x)}F_n(x){\rm d}x\right]{\rm d}u\\
&=&[1+T_n(t)]\widetilde{\textbf{T}}_sF_n(t)+\int_{0}^{t}e^{-s(t-u)}\tau_n(u)\widetilde{\textbf{T}}_sF_n(u){\rm d}u,
\end{eqnarray*}
and the proof is completed. 
\end{proof}
\par
The following corollary can be obtained from (\ref{past:ce:mean:inactivity}) and (\ref{T:tilda:q:n:s}) by setting $s=0$.
\begin{corollary}
Under the conditions of Theorem \ref{thm:T:tilda:q:n:s}, we have, for $t>0$,
\begin{equation}\label{colo:T:operator}
\widetilde{\textbf{T}}_0F_{n+1}(t)=F_n(t)[\widetilde{\mu}_n(t)(1+T_n(t))+\mathcal{CE}(X_n;t)],\qquad n\geq1.
\end{equation}
\end{corollary}
\subsection{Computational results}
\begin{figure}
\centering
\subfigure[]{\includegraphics[scale=0.49]{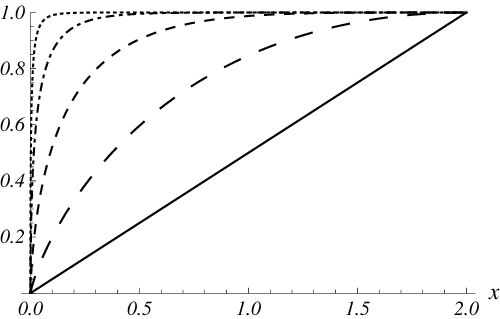}}
\;
\subfigure[]{\includegraphics[scale=0.49]{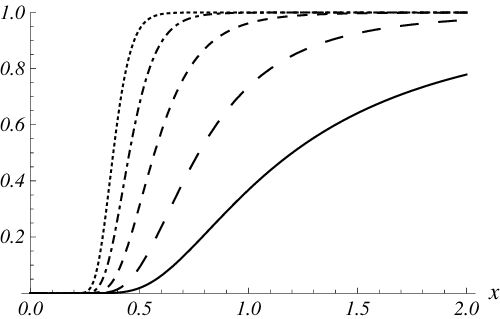}}
\;
\subfigure[]{\includegraphics[scale=0.49]{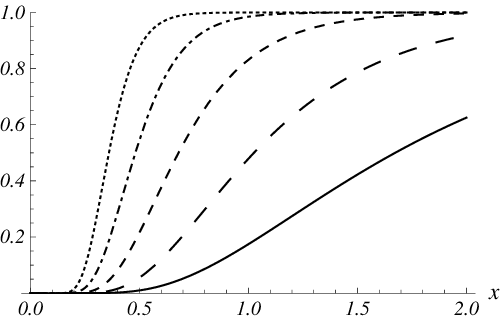}}
\\
\subfigure[]{\includegraphics[scale=0.49]{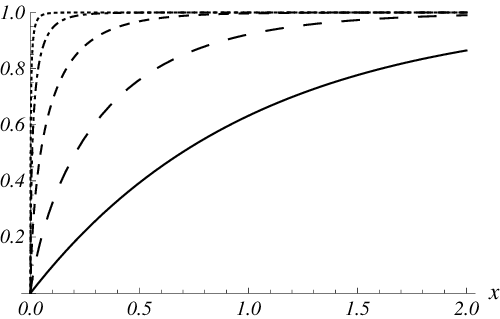}}
\;
\subfigure[]{\includegraphics[scale=0.49]{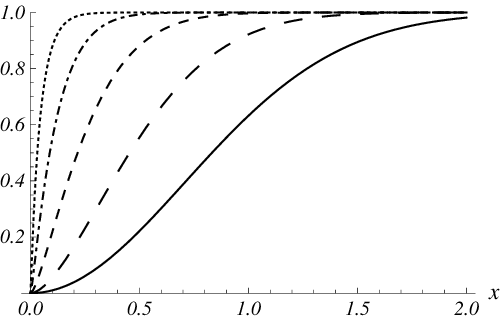}}
\;
\subfigure[]{\includegraphics[scale=0.49]{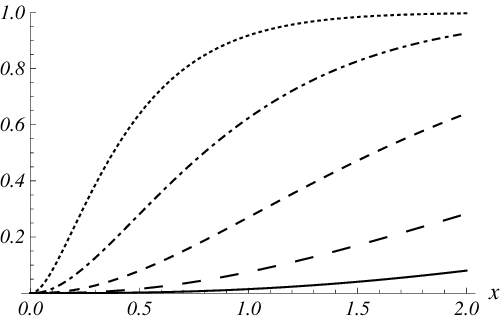}}
\caption{CDFs of $X_1,X_2,\ldots,X_5$ when $X_1$ follows 
the distributions given in Table 1.}
\label{}
\end{figure}
\begin{table}[t]
\centering
\caption{Starting distribution functions. }
\begin{tabular}{lccccccc}\hline
(a): \ $F(x)= {x}/{2},\quad 0< x< 2$\\
(b): \ $F(x)={\rm e}^{-x^{-2}}, \quad x>0$ \\
(c): \ $F(x)={\rm e}^{-3/({\rm e}^{x}-1)}, \quad x>0$ \\
(d): \ $F(x)=1-{\rm e}^{-x}, \quad x>0$ \\
(e): \ $F(x)=1-{\rm e}^{-x^2}, \quad x>0$ \\
(f): \ $F(x)=1-{\rm e}^{-x/2}\left(1+ {x}/{2}+ {x^2}/{8}\right), \quad x>0$ \\
\hline
\end{tabular}
\end{table}
We conclude the paper with few illustrative examples which shed some light on the behavior of the sequences of random variables  defined in (\ref{sequence}).
\begin{example}\rm
Let $X_1$ be uniformly distributed on $[0,2]$. Then we have $T_1(t)=-\log t/2$,  $\widetilde{\mu}_1(t)=t/2$ 
and $\mathcal{CE}(X_1;t)=t/4$, $0<t<2$ (see \cite{Di-Longobardi-2009}), so that
$$
\mathbb E\left (\frac{T_1(X_1)}{\tau(X_1)}\,\Big|\,X_1\leq t\right)=\mathcal{CE}(X_1;t)+\tilde{\mu}_1(t)T_1(t)=\frac{t}{4}-\frac{t}{2}\log \frac{t}{2},
$$
and
$$
{\rm cov}(X_1,T_1(X_1)\,|\,X_1\leq t)
 =T_1(t)[\mu_1(t)-\mathbb E(X_1)]-\mathcal{CE}(X_1;t)=\frac{2-t}{2}\log \frac{t}{2}-\frac{t}{4}.
$$
From (\ref{cov:Xn:Xn+1}), we obtain
$$
{\rm cov}(X_1,X_2)={\rm cov}(X_1,X_1-\tilde{\mu}_1(X_1))=\frac{1}{2}{\rm var}(X_1)=\frac{1}{6}.
$$
Finally, from (\ref{colo:T:operator}), we obtain
$$
\widetilde{\textbf{T}}_0F_{2}(t)=F_1(t)[\widetilde{\mu}_1(t)(1+T_1(t))+\mathcal{CE}(X_1;t)] 
=\frac{t^2}{4}\left[\frac{3}{2}-\log\frac{t}{2}\right] 
\qquad \hbox{for $0<t<2$.}
$$
\end{example} 
\par
It is difficult to obtain neat analytical results for the sequence of random variables $\{X_n,n\geq1\}$ and, therefore, we are forced to proceed via numerical computations. To this aim, in Figure 1 we show plots of the cumulative distribution of the random variables $X_1,X_2,\ldots,X_5$ for different starting distribution functions that are listed in Table 1. In the figure, the solid line corresponds to the CDF of  $X_1$, the large-dashed line corresponds to the CDF of  $X_2$ and so on. Moreover, we compute numerically the mean of the recursive random variables as well as the corresponding cumulative entropy; see Tables 2 and 3, respectively.  Recalling (\ref{E:X:n+1}),  the cumulative entropy is computed as the difference of two consecutive means. As expected,  the mean  of $X_n$ decreases when $n$ increases, whereas  the cumulative entropy decreases when $n$ increases. According to the numerical findings, we expect that the cumulative entropies are decreasing for any given starting distribution function.
\begin{table}\label{table:gsp}
\centering
\caption{The mean of the starting distribution given in Table 1.}
\begin{tabular}{c|cccccc}\hline
{} & (a) & (b) & (c) & (d) & (e) & (f)\\ 
\hline
$\mathbb E(X_1)$ & 1.000\,00 & 1.772\,45 &1.937\,91 & 1.000\,00 & 0.886\,22& 6.000\,00 \\
$\mathbb E(X_2)$ & 0.500\,00 & 0.886\,62 &1.151\,66 & 0.355\,06 & 0.506\,59& 3.359\,35 \\
$\mathbb E(X_3)$ & 0.180\,66 & 0.606\,80 &0.739\,45 & 0.104\,92 & 0.260\,35& 1.820\,59 \\
$\mathbb E(X_4)$ & 0.047\,13 & 0.469\,75 &0.511\,03 & 0.024\,83 & 0.118\,26& 0.944\,89 \\
$\mathbb E(X_5)$ & 0.009\,06 & 0.387\,74 &0.375\,44 & 0.004\,59 & 0.047\,17& 0.466\,17 \\
\hline
\end{tabular}
\end{table}
\begin{table}
\centering
\caption{The cumulative entropy of the starting distribution given in Table 1.}
\begin{tabular}{c|cccccc}\hline
{} & (a) & (b) & (c) & (d) & (e) & (f)\\ \hline
$\mathcal{CE}(X_1)$ & 0.500\,00 & 0.886\,23 &0.786\,25 & 0.644\,94 & 0.379\,63& 2.640\,65 \\
$\mathcal{CE}(X_2)$ & 0.319\,34 & 0.279\,35 &0.412\,21 & 0.250\,14 & 0.246\,24& 1.538\,76 \\
$\mathcal{CE}(X_3)$ & 0.133\,53 & 0.137\,12 &0.228\,42 & 0.080\,09 & 0.142\,09& 0.875\,70 \\
$\mathcal{CE}(X_4)$ & 0.038\,07 & 0.082\,01 &0.135\,59 & 0.020\,24 & 0.071\,09& 0.475\,72 \\
\hline
\end{tabular}
\end{table}
\section*{Acknowledgements}
We wish to thank and acknowledge partial support from the Ordered and Spatial Data Center of Excellence of Ferdowsi University of Mashhad of Iran, from GNCS-INdAM, and Regione Campania. Abdolsaeed Toomaj is grateful to the Department of Mathematics, Salerno University, for the hospitality during a four month visit in 2014.
\par
We gratefully thank the anonymous referee for his/her careful reading of the paper and for valuable suggestions. 
\end{document}